\documentclass[10pt]{amsart}
\usepackage{graphicx,enumitem,dsfont}
\usepackage[usenames]{color}
\usepackage[colorlinks]{hyperref}
\usepackage{subfigure}

\newcommand{\N}{{\mathbb{N}}}
\newcommand{\R}{\mathbb{R}}

\newcommand{\bD}{\mathfrak{d}}
\newcommand{\bP}{\mathbf{P}}
\newcommand{\mL}{\mathcal{L}}
\newcommand{\mR}{\mathcal{R}}
\newcommand{\mD}{\mathcal{D}}
\newcommand{\mU}{\mathcal{U}}

\newcommand{\bV}{\mathbf{V}}

\newcommand{\e}[1]{\ensuremath{10^{#1}}}

\newcommand{\U}{{\mathbf{u}}}
\newcommand{\B}{{\mathbf{B}}}

\newcommand{\X}{{\mathbf{x}}}

\newcommand{\g}{{\mathbf{g}}}

\newcommand{\Div}{\mathrm{div}}
\newcommand{\Curl}{\mathrm{curl}}

\newcommand{\diag}[1]{\mathrm{diag}\left(#1\right)}

\newcommand{\Comment}[1]{}

\renewcommand{\i}{\ifmmode\mathit{\mathchar"7010 }\else\char"10 \fi}
\renewcommand{\j}{\ifmmode\mathit{\mathchar"7011 }\else\char"11 \fi}

\newcommand{\seq}[1]{\left\{#1\right\}}

\newcommand{\Dx}{\Delta x}
\newcommand{\Dy}{\Delta y}
\newcommand{\Dt}{\Delta t}

\newcommand{\norm}[1]{\left\|#1\right\|}
\newcommand{\abs}[1]{\left|#1\right|}

\newcommand{\charf}{\mathbf{1}}

\newtheorem{theorem}{Theorem}[section]

\newtheorem{lemma}{Lemma}[section]

\newtheorem{remark}{Remark}[section]
\theoremstyle{definition} 

\newtheorem*{maintheorem*}{Main Theorem}

\allowdisplaybreaks

\numberwithin{equation}{section}
\numberwithin{figure}{section}
\numberwithin{table}{section}

\newcounter{asnr}

{\ifnum\value{asnr}=0 \stepcounter{asnr} 
  \begin{enumerate}[label=\textbf{A}.\arabic{enumi}]
    \else
    \begin{enumerate}[label=\textbf{A}.\arabic{enumi},resume] \fi}
{\end{enumerate}}

\title[SBP-SAT schemes for Magnetic induction equations]{Implicit finite difference schemes for the Magnetic Induction equations}

\author[U. Koley]{U. Koley} \address[Ujjwal Koley]{\newline Centre of Mathematics for Applications (CMA)
  \newline University of Oslo\newline P.O. Box 1053, Blindern\newline
  N--0316 Oslo, Norway} \email[]{ujjwalk@cma.uio.no}

\keywords{Conservation laws; Induction equations; Summation by parts operators; Simaltaneous approximation term; Fully-discrete schemes; High order of accuracy.}
\date{\today}

\begin{document}

\begin{abstract}
  We describe high order accurate and stable fully-discrete finite difference schemes
  for the initial-boundary value problem associated with the magnetic
  induction equations. These equations model the evolution of a
  magnetic field due to a given velocity field. The finite difference
  schemes are based on Summation by Parts (SBP) operators for spatial
  derivatives and a Simultaneous Approximation Term (SAT) technique
  for imposing boundary conditions. We present various numerical
  experiments that demonstrate both the stability as well as high
  order of accuracy of the schemes.
\end{abstract}

\maketitle

\section{Introduction}
\label{sec:intro}
In this paper, we study the magnetic induction equation
\begin{equation}
  \label{eq:induction}
 \partial_t \B + \Curl(\B\times\U) =  - \U \Div(\B), 
\end{equation}
where the unknown $\B = \B(\X,t) \in \R^3$ describes the magnetic field 
of a plasma in three space dimensions with coordinate $\X =(x,y,z)$.
The above equation models the evolution of the magnetic field in the plasma 
which is moving with a prescribed velocity field $\U(\X,t)$.
These equations arise in a wide
variety of applications in plasma physics, astrophysics and electrical
engineering. One important application are the equations of
magneto-hydro dynamics (MHD).
Observe that, by taking divergence on both sides of
(\ref{eq:induction}) we get
\begin{equation}
  \label{eq:divt}
  (\Div(\B))_t + \Div\left(\U\Div(\B)\right) = 0.
\end{equation}
Hence, if $\Div(\B_0(\X))=0$, also $\Div(\B(\X,t))=0$ for $t>0$.

There are many forms of induction equations available in literature (see \cite{uk, fkrsid1}).
Here we are going to work with the following ``conservative'' symmetric form,
\begin{equation}
  \begin{aligned}
    \partial_t\B + \left(\U\cdot\nabla\right)\B = M(D\U)\B,
  \end{aligned}
  \label{eq:induc1}  
\end{equation}
where the $D\U$ denotes the gradient of $\U$ and the matrix $M(D\U)$ is given by
$$
M(D\U)=
 \begin{pmatrix}
   -\partial_y u^2 - \partial_z u^3 & \partial_y u^1 & \partial_z
   u^1\\
   \partial_x u^2 & -\partial_x u^1-\partial_z u^3 &+\partial_z u^2\\
   \partial_x u^3 & \partial_y u^3 & -\partial_x u^1 - \partial_y u^2
  \end{pmatrix}.
$$
We are aware of some other results in the literature related to induction 
equation \cite{BeKr1, BlBr1, DW1, fkrsid1, TF1, uk1}.   
But, boundary conditions were not considered either
of the aforementioned papers.
In \cite{uk}, authors have described a high order accurate and stable finite
difference schemes for the initial-boundary problem associated with
the magnetic induction equations. The approach is based on a ``semi-discrete''
approximation where one discretizes the spatial variable, thereby reducing
the equations to a system of ordinary differential equations. However, we 
stress that for numerical computations also this set of ordinary 
differential equations will have to be discretized in order to be solved.
Thus in order to have a completely satisfactory numerical method, one seeks
a fully discrete scheme that reduces the actual computation to a solution
of a finite set of algebraic equations. 

Our aim in this paper is to design stable and high-order
accurate ``fully-discrete'' schemes for initial-boundary value problems corresponding to
the magnetic induction equations by discretizing the non-conservative symmetric 
form (\ref{eq:induc1}). The spatial derivatives are
approximated by second and fourth-order SBP (Summation-By-Parts)
operators. The boundary conditions are weakly imposed by using a SAT
(Simultaneous Approximation Term) and backward Euler method used for
temporal discretization. The SBP-SAT framework has been used to obtain stable and 
accurate high order schemes for a wide variety of hyperbolic problems in recent years. 
See \cite{SN1} and the references therein for more details.

The rest of this paper is organized as follows: In
Section~\ref{sec:cp}, we state the energy estimate for the
initial-boundary value problem corresponding to (\ref{eq:induc1}. In
Section~\ref{sec:scheme}, we present the fully-discrete SBP-SAT scheme and show stability. 
Numerical experiments are presented in
Section~\ref{sec:numex} and conclusions are drawn in
Section~\ref{sec:conc}.

\section{The Continuous problem}
\label{sec:cp}
For ease of notation, we shall restrict ourselves to two spatial
dimensions in the remainder of this paper. Extending the results to
three dimensions is straightforward.

In two dimensions, the non-conservative symmetric form
(\ref{eq:induc1}) reads
\begin{equation}
  \label{eq:main}
  \begin{aligned}
    \B_t + \Lambda_1\B_x + \Lambda_2\B_y - C\B &= 0,
    \end{aligned}
\end{equation}
where
\begin{align*}
  \Lambda_1&=
  \begin{pmatrix}
    u^1 & 0 \\
    0 & u^1
  \end{pmatrix},\quad
  \Lambda_2=
  \begin{pmatrix}
    u^2 & 0 \\
    0 & u^2
  \end{pmatrix},\quad
  C=\begin{pmatrix}
    -\partial_y u^2 & \partial_y u^1 \\
    \partial_x u^2 & -\partial_x u^1
  \end{pmatrix},
\end{align*}
with $\B = \left( B^1, B^2\right)^T$ and $\U = \left(u^1,u^2\right)^T$
denoting the magnetic and velocity fields respectively. In component
form, (\ref{eq:main}) becomes
\begin{equation}
  \label{eq:NC2D}
  \begin{aligned}
    (B^1)_t + u^1(B^1)_x + u^2(B^1)_y &= -(u^2)_y B^1 + (u^1)_y B^2 \\
    (B^2)_t + u^1(B^2)_x + u^2(B^2)_y &= (u^2)_x B^1 - (u^1)_x B^2.
  \end{aligned}
\end{equation}
To begin with, we shall consider (\ref{eq:main}) in the domain
$(x,y)\in\Omega=[0,1]^2$.

We augment (\ref{eq:main}) with initial conditions,
\begin{equation}
  \label{eq:initial}
  \begin{aligned}
    \B(\X,0) = \B_0(\X) \quad \X\in \Omega,
 \end{aligned}
\end{equation}
and Dirichlet boundary conditions,
\begin{equation}
  \label{eq:boundary}
  \begin{aligned}
    \charf_{\seq{u^1(0,y,t)>0}}\Bigl(\B(0,y,t) &= \g(0,y,t)\Bigr),
    \quad
    \charf_{\seq{u^1(1,y,t)<0}} \Bigl(\B(1,y,t) = \g(1,y,t)\Bigr),\\
    \charf_{\seq{u^2(x,0,t)>0}}\Bigl(\B(x,0,t) &=\g(x,0,t)\Bigr),
    \quad \charf_{\seq{u^2(x,1,t)<0}}\Bigl(\B(x,1,t) = \g(x,1,t)\Bigr)
  \end{aligned}
\end{equation}
where $\charf_{A}$ denotes the characteristic function of the set $A$.
Note that we only impose boundary conditions on the set where the
characteristics are entering the domain.
We shall always assume that the initial and boundary data satisfy the
compatibility conditions, i.e., specific criteria that guarantee smoothness 
of the solution, see \cite{GustafssonKreissOliger}.
\begin{theorem}
  \label{theo:cp}
  Assume that $\B_0\in H^1(\Omega)$, that $\g\in
  H^1(\partial\Omega\times [0,T])$ for $T>0$ and that $u^1$ and $u^2$
  are in $H^2(\Omega\times[0,T])$. Then there exists a function $\B \in
  C([0,T],L^2(\Omega))\cap L^\infty([0,T];H^1(\Omega))$ which is the
  unique weak solution of (\ref{eq:main}) with the initial and
  boundary conditions (\ref{eq:initial}) and (\ref{eq:boundary}).

  Furthermore, it satisfies the following stability estimate
  \begin{equation}
    \label{eq:enest}
    \norm{\B(\cdot,t)}_{H^1(\Omega)}^2\le e^{\alpha t}\left(\norm
    {\B_0}_{H^1(\Omega)}^{2} + \norm{\g}_{H^1(\partial\Omega\times(0,t))}\right).
  \end{equation}
  where $\alpha$ is a positive constant.
\end{theorem}

\section{Fully-discrete Scheme}
\label{sec:scheme}
To simplify the treatment of the boundary terms we let 
the computational domain $\Omega$ be the unit square. It is straightforward to generalize our results
to other domains by coordinate transformations (see \cite{Svard04}), and to three dimensions. 

The SBP finite difference schemes for one-dimensional derivative approximations are as follows. Let $[0,1]$ be the domain discretized with $x_j=j\Dx$, $j=0,\dots,N-1$. A scalar grid function is defined as $w=(w_0,...w_{N-1})$. To approximate
$\partial_x w$ we use a summation-by-parts operator $D_x = P^{-1}_x
Q_x$, where $P_x$ is a diagonal positive $N\times N$ matrix, defining an inner
product
$$
(v,w)_{P_x} = v^T P_x w,
$$
such that the associated norm $\norm{w}_{P_x} = (w,w)_{P_x}^{1/2}$
is equivalent to the norm $\norm{w}=(\Dx\sum_{k} w_k^2)^{1/2}$. 
Furthermore, for $D_x$ to be a summation-by-parts operator we require
that
$$
Q_x + Q_x^T = R_N-L_N,
$$
where $R_N$ and $L_N$ are the $N\times N$ matrices: $\diag{0,\dots,1}$
and $\diag{1,\dots,0}$ respectively.  Similarly, we can define a
summation-by-parts operator $D_y=P_y^{-1}Q_y$ approximating
$\partial_y$. Later we will also need the following Lemma, proven in \cite{svardsid3}.
\begin{lemma}\label{lemma:variable}
Let $u$ be a smooth grid function. Then
\begin{equation}
 \label{eq:circdiff}
 \begin{aligned}
   \norm{D_x(u\circ w)- u\circ D_x w}_{P_x} \le C \norm{\partial_x u}_{L^\infty([0,1])} \norm{w}_{P_x}
 \end{aligned}
\end{equation}
where $(u\circ v)_j=u_jv_j$.
\end{lemma}

Next, we move on to the two-dimensional case and discretize the unit square $[0,1]^2$ using $NM$ uniformly
distributed grid points $(x_i,y_j)=(i\Dx,j\Dy)$ for $i=0,\dots,N-1$,
and $j=0,\dots,M-1$, such that $(N-1)\Dx=(M-1)\Dy=1$. We order a scalar
grid function $w(x_i,y_i)=w_{ij}$ as a column vector
$$
w=\left(w_{0,0},w_{0,1},\dots,
  w_{0,(M-1)},w_{1,0},\dots,\dots,w_{(N-1),(M-1)}\right)^T.   \label{eq:w}
$$ 

To obtain a compact notation for partial derivatives of a grid
function, we use Kronecker products. The Kronecker product of an
$N_1\times N_2$ matrix $A$ and an $M_1\times M_2$ matrix $B$ is
defined as the $N_1M_1\times N_2M_2$ matrix
\begin{equation}
\label{eq:order}
A\otimes B=
\begin{pmatrix}
  a_{11}B &\dots & a_{1N_2}B\\
  \vdots  & \ddots & \vdots \\
  a_{N_1 1}B  & \dots& a_{N_1 N_2}B  
\end{pmatrix}.
\end{equation}
For appropriate matrices $A$, $B$, $C$ and $D$, the Kronecker product
obeys the following rules:
\begin{align}
  (A\otimes B)(C\otimes D)&=(AC\otimes BD),\\
 (A\otimes B)+(C\otimes
  D)&=(A+C)\otimes(B+D),\\
   (A\otimes B)^T&=(A^T\otimes B^T).
\end{align}

Using Kronecker products, we can define 2-D difference operators. Let $I_n$ denote the $n\times n$ identity matrix, and define
$$
\bD_x = D_x \otimes I_M, \quad \bD_y = I_N \otimes D_y.
$$
For a smooth function $w(x,y)$, $(\bD_x w)_{i,j}\approx \partial_x
w(x_i,y_j)$ and similarly $(\bD_y w)_{i,j}\approx \partial_y
w(x_i,y_j)$. 

Set $\bP=P_x\otimes P_y$, define $(w,v)_{\bP}=w^T\bP v$ and the
corresponding norm $\norm{w}_{\bP}=(w,w)_{\bP}^{1/2}$. Also define 
$\mR=R_N\otimes I_M$, $\mL=L_N\otimes I_M$, $\mU=I_N \otimes R_M$ and 
$\mD=I_N \otimes L_M$. 

For a vector valued grid function $\mathbf{V}=(V^1,V^2)$, we use the following notation
$$
\bD_x \mathbf{V}=
\begin{pmatrix}
  \bD_x V_1\\ \bD_x V_2
\end{pmatrix},
$$
and so on. In the same spirit, the $\bP$ inner product of vector valued
grid functions is defined by $(\mathbf{V},\mathbf{W})_{\bP} = (V^1,W^1)_{\bP} + 
(V^2,W^2)_{\bP}$.
We also introduce (a small) time step $\Dt >0$, and use the notation
\begin{align*}
D_t^{+} p(t) = \frac{1}{\Dt} \left( p(t + \Dt) -p(t)\right),
\end{align*}
for any function $p: [0,T] \rightarrow \R$. Write $t^n = n\Dt$ for $n \in \N_0 = \N \cup \lbrace 0 \rbrace$.
We will use the notation $ V^1(x_i, y_j, t^n) = V_{ij}^{1,n}$ and so on.

\begin{remark}
Note that the Kronecker products is just a tool to facilitate the
notation. In the implementation of schemes using the operators in the
Kronecker products we can think of these as operating in their own
dimension, i.e., on a specific index. Thus, to compute $\bD_x w$, we
can view $w$ as a field with two indices, and the one-dimensional
operator $D_x$ will operate on the first index since it appears in the
first position in the Kronecker product. 
\end{remark}

The usefulness of summation by parts operators comes from this lemma.
\begin{lemma}
  \label{lem:byparts} 
  For any grid functions $v$ and $w$, we have
  \begin{equation}
    \label{eq:byparts}
    \begin{aligned}
      \left(v,\bD_xw\right)_{\bP} + \left(\bD_xv,w\right)_{\bP} &= 
      v^T \left[(\mR - \mL)(I_N \otimes P_y)\right] w\\ 
      \left(v,\bD_yw\right)_{\bP} + \left(\bD_yv,w\right)_{\bP} &= 
      v^T \left[(\mU - \mD)(P_x \otimes I_M)\right] w.
    \end{aligned}
  \end{equation}
\end{lemma}
Observe that this lemma is the discrete version of the equality
$$
\iint_\Omega v\left(\partial_x w\right)\,dxdy + \iint_\Omega
\left(\partial_x v\right) w \,dxdy =   
\int_0^1 v(1,y) w(1,y) - v(0,y)w(0,y) \,dy.
$$
\begin{proof}
  We calculate
  \begin{align*}
    \left(v,\bD_xw\right) &= v^T\left(P_x\otimes P_y\right) 
    \left(P_x^{-1}Q_x\otimes I_M\right) w \\
    &=v^T Q_x \otimes P_y w \\
    &= -v^T Q_x^T \otimes P_y w + v^T (Q_x + Q_x^T) \otimes P_y w \\
    &= - (P_x^{-1}Q_x\otimes I_M v)^T \left(P_x\otimes P_y \right) w +
    v^T\left(R_N - L_N\right)\otimes P_y w\\
    &= - \left(\bD_xv\right)^T \left(P_x\otimes P_y \right) w 
    + v^T \left(\mR-\mL\right)(I_N \otimes P_y) w.
  \end{align*}
  The second equality is proved similarly.
\end{proof}

Before we define our numerical schemes, we collect some useful results
in a lemma.
\begin{lemma} 
  \label{lem:kurlkurl} 
  If $u$ is a grid function, then 
  \begin{equation}
    \label{eq:boundaries}
    \begin{aligned}
      \left(\bV,u\circ\bD_x\bV\right)_{\bP} &= \frac{1}{2} \bV^T
      [(\mR-\mL)(I_N\otimes
      P_y)] \left(u\circ\bV\right)
      \\
      &\qquad +
      \frac{1}{2}\left(u\circ \bD_x \bV - \bD_x\left(u\circ \bV\right),
        \bV\right)_{\bP},
      \\
      \left(\bV,u\circ\bD_y\bV\right)_{\bP} &= \frac{1}{2} \bV^T
      [(\mU-\mD)(P_x\otimes I_M)] \left(u\circ\bV\right)
      \\
      &\qquad +
       \frac{1}{2}\left(u\circ \bD_y \bV - \bD_y\left(u\circ \bV\right),
        \bV\right)_{\bP}.
    \end{aligned}
  \end{equation}
\end{lemma}
\begin{proof}
  To show \eqref{eq:boundaries}, first note that since $\bP$ is diagonal, $(u\circ
  \bD \bV,\bV)_{\bP} = (\bD \bV,u\circ \bV)_{\bP}$. We use
  Lemma~\ref{lem:byparts} to calculate
  \begin{align*}
    \left(u\circ \bD_x \bV, \bV \right)_{\bP}
    &= \left(\bD_x(u\circ \bV\right),\bV)_{\bP} + 
    \left(u\circ \bD_x \bV - \bD_x\left(u\circ \bV\right), \bV\right)_{\bP} \\
    &= -\left(u\circ \bV, \bD_x \bV\right)_{\bP} + 
    \bV^T \left(\mR-\mL\right)\left(I_N\otimes P_y\right) (u\circ \bV)
    \\
    &\qquad +
    \left(u\circ \bD_x \bV - \bD_x\left(u\circ \bV\right), \bV\right)_{\bP}
    \\
     &= -\left(u\circ \bD_x \bV,\bV\right)_{\bP} + 
    \bV^T \left(\mR-\mL\right)\left(I_N\otimes P_y\right) (u\circ \bV)
    \\
    &\qquad +
    \left(u\circ \bD_x \bV - \bD_x\left(u\circ \bV\right), \bV\right)_{\bP}.
  \end{align*}
  This shows the first equation in \eqref{eq:boundaries}, the second
  is proved similarly.
\end{proof}
Now we are in a position to state our scheme(s). For $\ell=1$ or $2$
we will use the notation $u^\ell$ for both the grid function defined
by the function $u^\ell(x,y)$ and for the function itself. Similarly,
for the boundary values, we use the notation $h$ and $g$ for both
discrete and continuously defined functions. Hopefully,
it will be apparent from the context what we refer to. 

The differential equation \eqref{eq:main} will be discretized in an obvious manner. We incorporate the boundary conditions by penalizing boundary
values away from the desired ones with a $\mathcal{O}(1/\Dx)$ term. To
this end set 
$$
\mathcal{B}=\left[ \left(P_x^{-1}\otimes I_M\right)\left(\Sigma_{\mL}\mL +
    \Sigma_{\mR}\mR\right) + \left(I_N\otimes P_y^{-1}\right)
  \left(\Sigma_{\mD}\mD + \Sigma_{\mU}\mU\right)\right],
$$
where $\Sigma_{\mL}$, $\Sigma_{\mR}$, $\Sigma_{\mD}$ and
$\Sigma_{\mU}$ are diagonal matrices, with components $(\sigma_{\mL})_{jj}$  ordered in the same way as in (\eqref{eq:order}) (and similarly for the other penalty matrices), to be specified later. 
   
With this  notation the scheme for the differential equation
\eqref{eq:main} reads
\begin{equation}
  \label{eq:dirichletdiscrete}
  D_t^{+}\bV^n + u^{1,n+1}\circ \bD_x \bV^{n+1} + u^{2,n+1}\circ \bD_y \bV^{n+1} - C^{n+1} \bV^{n+1}  
  = \mathcal{B}(\bV^{n+1}-\mathbf{g}^{n+1}),
\end{equation}
while $\bV^0$ is given. Here $C^n$ denotes the matrix
$$
C^n=
\begin{pmatrix}
  -\bD_y u^{2,n} & \bD_y u^{1,n} \\ \bD_x u^{2,n} & -\bD_x u^{1,n} 
\end{pmatrix}.
$$

\begin{theorem}
  \label{thm:stab1} Let $\bV$ be as solution to
  \eqref{eq:dirichletdiscrete} with $\mathbf{g}=0$. If the constants
  in $\mathcal{B}$ is chosen as
\begin{equation}
  \begin{gathered}
    (\sigma_{\mR})_{N-1,j}\le \frac{u^{1,-}(1,y_j)}{2}, \ (\sigma_{\mL})_{0,j} \le  -\frac{u^{1,+}(0,y_j)}{2},
    \ (\sigma_{\mU})_{i,M-1}\le \frac{u^{2,-}(x_i,1)}{2}, \\
    \text{and}\ (\sigma_{\mD})_{i,0} \le -\frac{u^{2,+}(x_i,0)}{2},
  \end{gathered}
  \label{eq:primeconstants}
\end{equation} 
then 
  \begin{equation}
    \label{eq:stabdirichlet}
    \norm{\bV^{n}}^2_{\bP} \le e^{KT} \norm{\bV^{0}}_{\bP}^2,  
  \end{equation}
where $ u^{l,+} = (u^l \vee 0)$, $ u^{l,-} = (u^l \wedge 0)$, for $ l = 1,2$. $K$ is a constant chosen in such a way that 
$(1-c\Dt)^{-1} \le (1+K\Dt)$ for sufficiently small $\Dt$, where $c$ is a constant 
depending on $u^1$, $u^2$, and their derivative approximations, but not on $N$ or $M$.
\end{theorem}
\begin{proof}
Taking the $\bP$ inner product of
  \eqref{eq:dirichletdiscrete} and $\bV^{n+1}$, we get 
  \begin{align*}
&\frac{1}{2} \norm{\bV^{n+1}}_{\bP}^2  - \frac{1}{2} \norm{\bV^{n}}_{\bP}^2  + \frac{1}{2} \norm{\bV^{n+1} - \bV^{n}}_{\bP}^2 
 = - \Dt \left(\bV^{n+1},u^{1,n+1}\circ\bD_x \bV^{n+1}\right)_{\bP} \\
& - \Dt \left(\bV^{n+1},u^{2,n+1}\circ\bD_y \bV^{n+1}\right)_{\bP}
    + \Dt \left(\bV^{n+1},C^{n+1}\bV^{n+1}\right)_{\bP} + \Dt
    \left(\bV^{n+1},\mathcal{B}\bV^{n+1}\right)_{\bP}.
   \end{align*}
Using Lemma~\ref{lem:kurlkurl} we get
   \begin{align*}
\frac{1}{2} \norm{\bV^{n+1}}_{\bP}^2  &-\frac{1}{2} \norm{\bV^{n}}_{\bP}^2  + \frac{1}{2} \norm{\bV^{n+1} - \bV^{n}}_{\bP}^2\\ & = - \Dt \frac{1}{2} (\bV^{n+1})^T\left[ (\mR-\mL)(I_N\otimes
      P_y)\right](u^{1,n+1}\circ \bV^{n+1}) \\
& - \Dt \frac{1}{2} \b(V^{n+1})^T\left[
      (\mU-\mD)(P_x\otimes I_M)\right](u^{2,n+1}\circ \bV^{n+1})
    \\
    & - \Dt \frac{1}{2}\left(u^{1,n+1}\circ \bD_x \bV^{n+1} - \bD_x(u^{1,n+1}\circ
      \bV^{n+1}),\bV^{n+1}\right)_{\bP} \\
& - \Dt \frac{1}{2}\left(u^{2,n+1}\circ \bD_y \bV^{n+1} -
      \bD_y(u^{2,n+1}\circ \bV^{n+1}),\bV^{n+1}\right)_{\bP}
    \\
    &\qquad +\Dt \left(\bV^{n+1},C\bV^{n+1}\right)_{\bP} + \Dt
    \left(\bV^{n+1}, \mathcal{B}\bV^{n+1}\right)_{\bP}.
  \end{align*}
Note that by \eqref{eq:circdiff}, 
\begin{equation}\label{eq:squareterms}
 \begin{aligned}
   \abs{ \left(u^{1,n+1}\circ \bD_x \bV^{n+1} - \bD_x(u^{1,n+1}\circ
       \bV^{n+1}),\bV^{n+1}\right)_{\bP}}&\le  c \norm{\bV^{n+1}}_{\bP}^2,\\
   \abs{\left(u^{2,n+1}\circ \bD_y \bV^{n+1} - \bD_y(u^{2,n+1}\circ
       \bV^{n+1}),\bV^{n+1}\right)_{\bP}}&\le  c \norm{\bV^{n+1}}_{\bP}^2,\\
   \abs{\left(\bV^{n+1},C\bV^{n+1}\right)_{\bP}}&\le c \norm{\bV^{n+1}}_{\bP}^2,
 \end{aligned}
\end{equation}
for some constant $c$ depending on the first derivatives of $u^1$
and $u^2$.  Using the conditions (\ref{eq:primeconstants}) we arrive at
\begin{align*}
\norm{\bV^{n+1}}_{\bP}^2 \le \norm{\bV^{n}}_{\bP}^2 + c \Dt \norm{\bV^{n+1}}_{\bP}^2    
\end{align*}
Now we can use the fact that $(1-c\Dt)^{-1} \le (1+K\Dt)$ for sufficiently small $\Dt$.
Consequently this gives the required bound \eqref{eq:stabdirichlet}.
\end{proof}

\section{Numerical Experiment}
\label{sec:numex}

We test the fully-discrete SBP-SAT scheme of the previous section on a suite of
numerical experiments in order to demonstrate the effectiveness of
these schemes. We will use two different schemes : $SBP2$ and $SBP4$ scheme which 
are second-order (first-order) and fourth order (second-order) accurate in 
the interior (boundary) resulting in an overall second and third-order of accuracy.

\vspace{0.5cm}

In this experiment, we consider (\ref{eq:main}) with the
divergence-free velocity field $\U(x,y) = (-y,x)^T$. The
exact solution can be easily calculated by the method of
characteristics and takes the form
\begin{equation}
  \label{eq:ex}
  \B(\X,t)=R(t)\B_0(R(-t)\X),
\end{equation}
where $R(t)$ is a rotation matrix with angle $t$ and represents rotation of the initial data about the origin.

We consider the same test setup as in \cite{TF1} and \cite{fkrsid1} by
choosing the divergence free initial data,
\begin{equation}
  \B_0(x,y)=4
  \begin{pmatrix}
    -y\\ x-\frac{1}{2}
  \end{pmatrix}
  e^{-20\left((x-1/2)^2+y^2\right)},
\end{equation}
and the computational domain $[-1,1] \times [-1,1]$. Since the exact solution is known in
this case, one can in principle use this to specify the
boundary data $g$. Instead, we decided to
mimic a free space boundary (artificial boundary) by taking $g=0$. (which is a good guess at a far-field boundary).

We run this test case with $SBP2$ and $SBP4$ schemes and present
different sets of results. In Figure \ref{fig:1}, we plot $|\B|=(|B^1|^2+|B^2|^2)^{1/2}$ at times $t= \pi$ (half-rotation) and $t=2\pi$ (one
full rotation) with the $SBP2$ and $SBP4$ schemes. 
\begin{figure}[htbp]
  \centering
\subfigure[half rotation, SBP2]{    \includegraphics[width=0.45\linewidth]{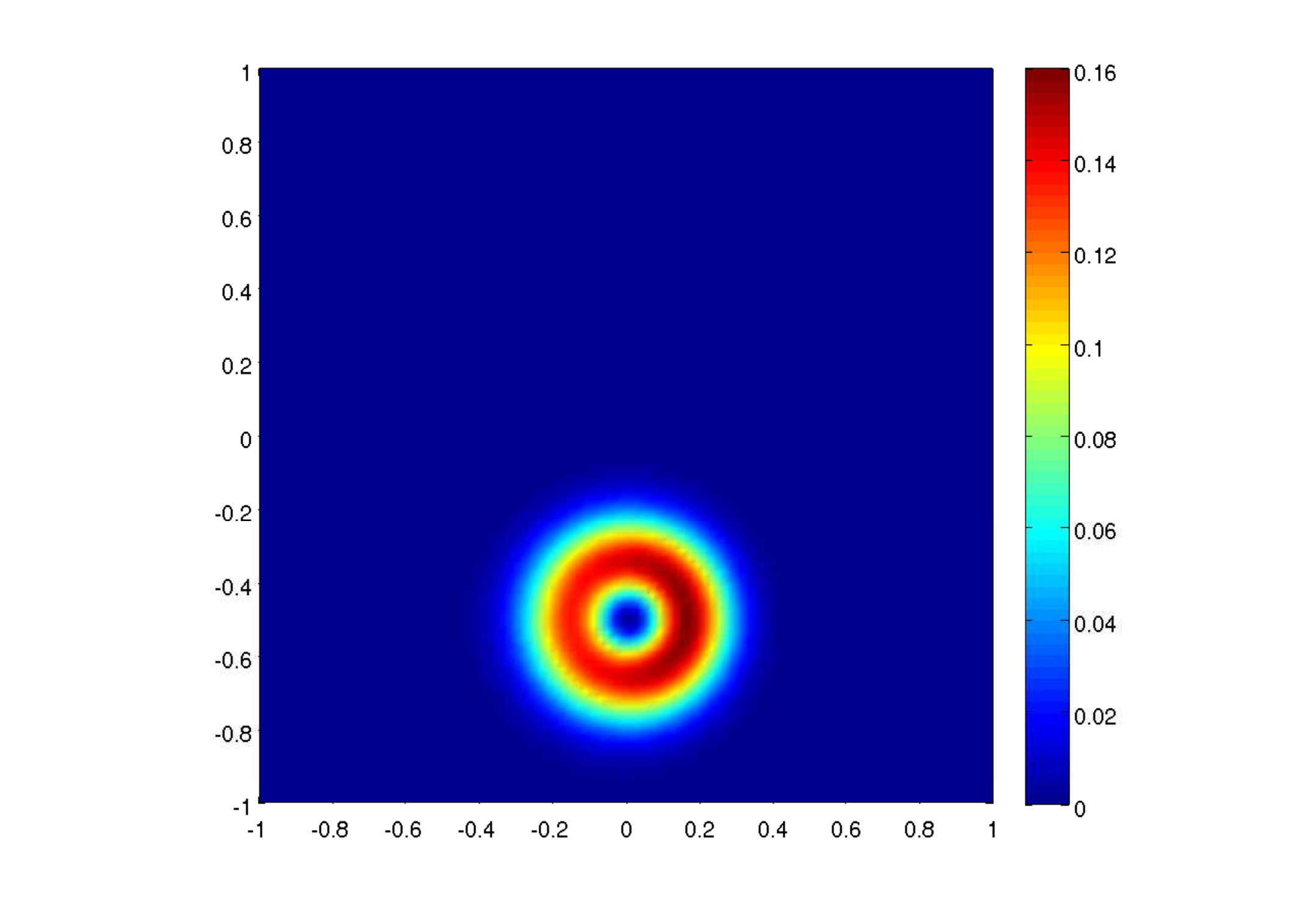} }
\subfigure[full rotation, SBP2]{    \includegraphics[width=0.45\linewidth]{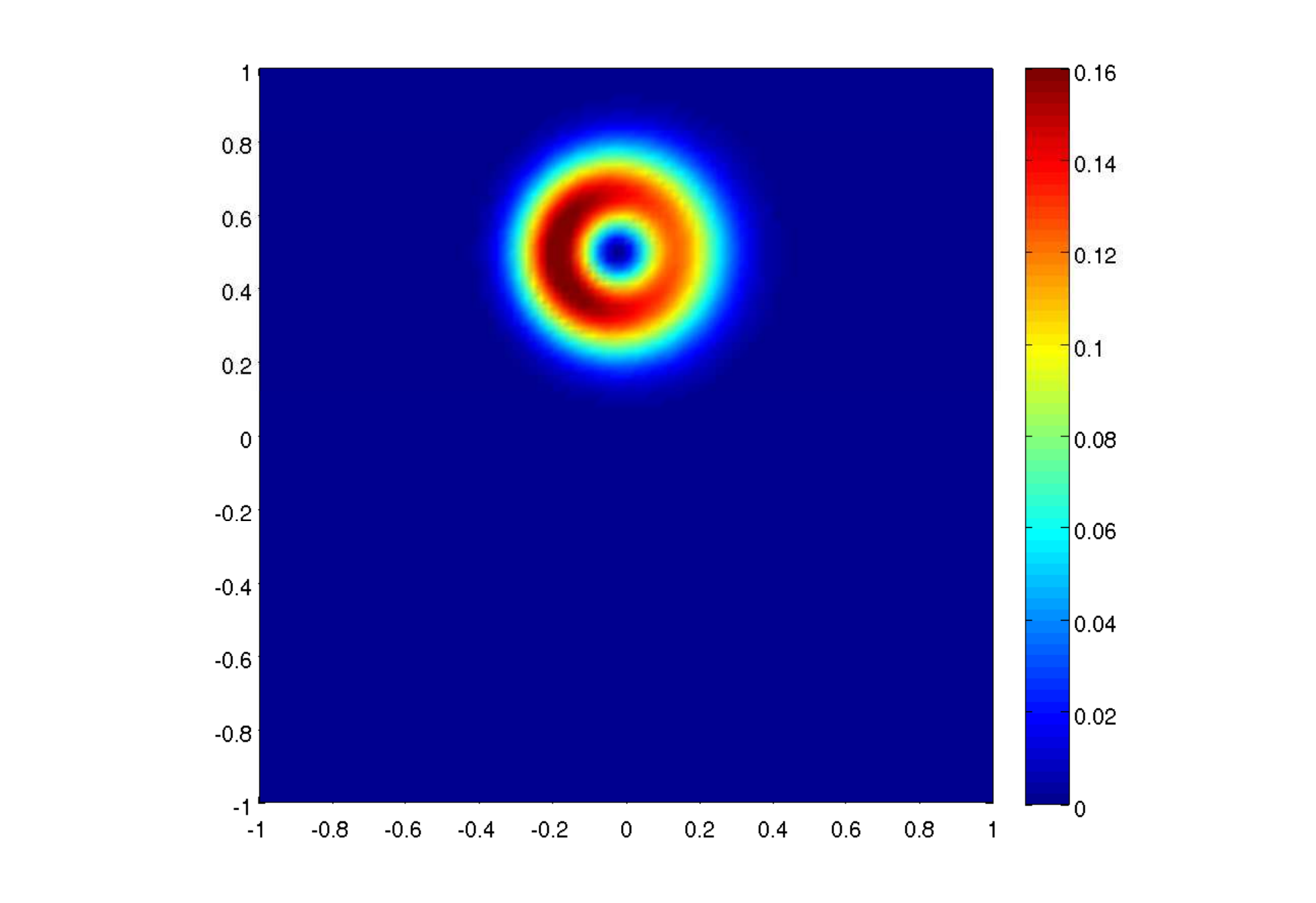}}\\
\subfigure[half rotation, SBP4]{    \includegraphics[width=0.45\linewidth]{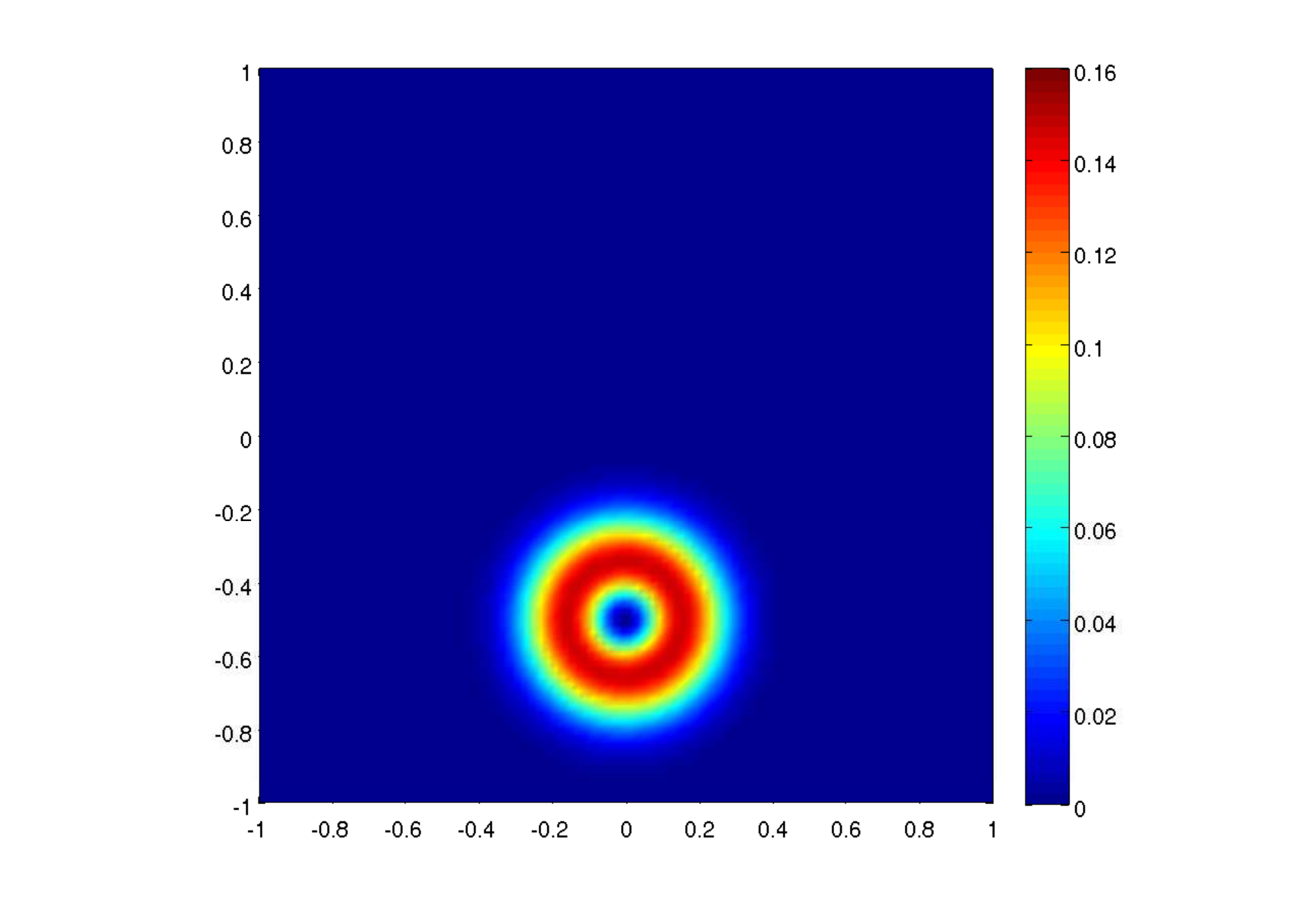}}
\subfigure[full rotation, SBP4]{    \includegraphics[width=0.45\linewidth]{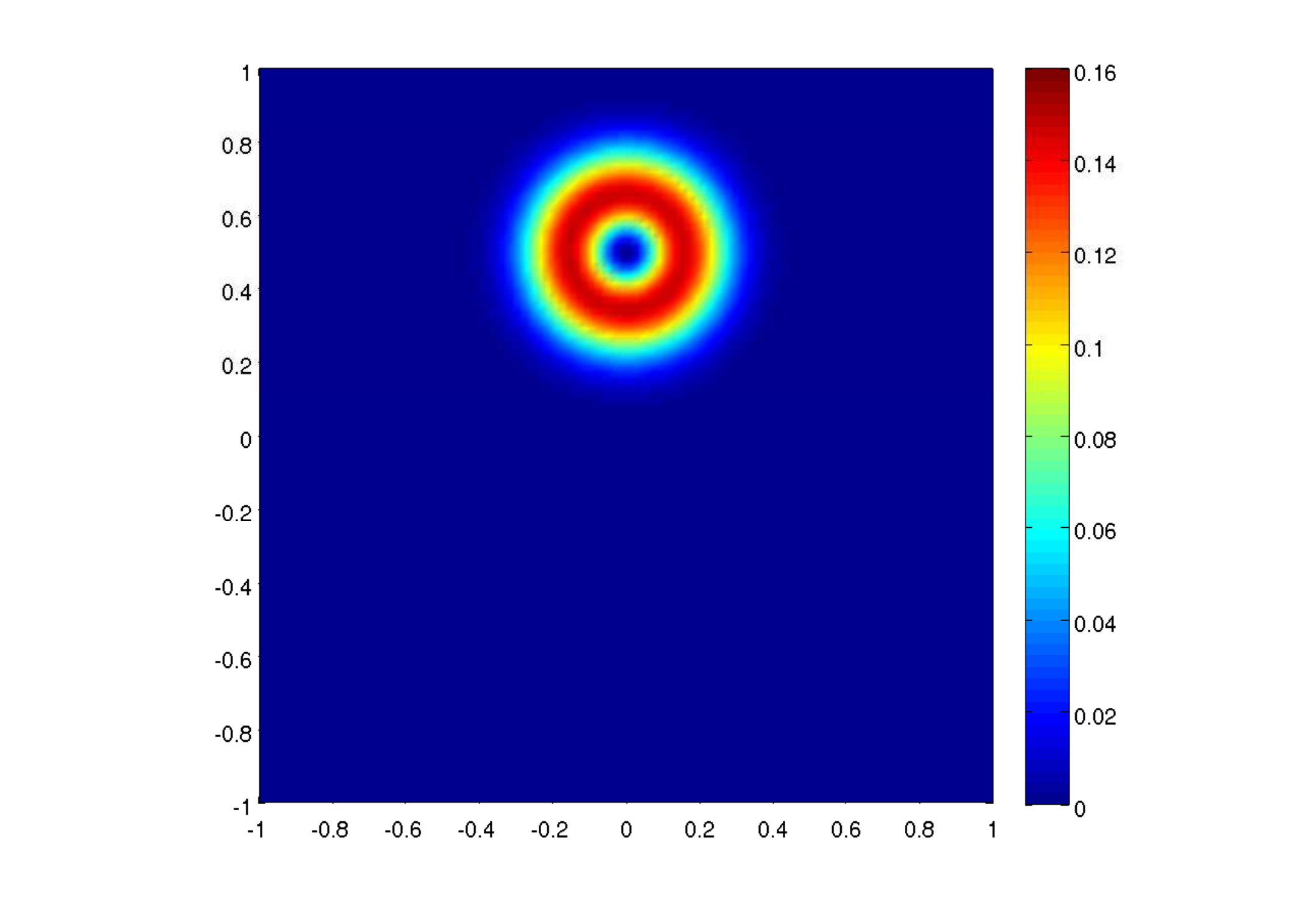} }
 \caption{Numerical results for $|\B|$.}
\protect \label{fig:1}  
\end{figure}
As shown in this figure, $SBP2$ and $SBP4$ schemes resolve the
solution quite well. In fact, $SBP4$ is very accurate and keeps the
hump intact throughout the rotation.
\begin{table}[htbp]
\centering
\begin{tabular}{c|rrrr}
Grid size & $SBP2$ & rate & $SBP4$ & rate\\
\hline
40$\times$40   & 6.9 $\times$\e{1}     &       & 8.0 $\times$\e0     &    \\
80$\times$80   & 2.1 $\times$\e{1}     & 1.7   & 5.0 $\times$\e{-1}  & 4.0 \\
160$\times$160 & 5.5 $\times$\e{0}     & 2.0   & 4.5 $\times$\e{-2}  & 3.5 \\
320$\times$320 & 1.3 $\times$\e{0}     & 2.0   & 5.1 $\times$\e{-3}  & 3.1 \\
640$\times$640 & 3.3 $\times$\e{-1}    & 2.0   & 6.4 $\times$\e{-4}  & 3.0
\end{tabular}
\caption{Relative percentage errors in $l^2$ for $|\B|$ at time $t = 2
  \pi$ and rates of convergence with $SBP2$ and $SBP4$ schemes.}
\label{tab:1}
\end{table}
In Table \ref{tab:1}, we present percentage relative errors in $l^2$. The
errors are computed at time $t = 2 \pi$ (one rotation) on a sequence
of meshes for both the $SBP2$ and $SBP4$ schemes. The results show
that the errors are quite low, particularly for $SBP4$ and the rate of convergence approaches the expected values of $2$ for $SBP2$ and $3$ for $SBP4$. Furthermore, the order of accuracy is
unaffected at these resolutions by using zero Dirichlet boundary data instead of the exact
solution at the boundary.

\section{Conclusion}
\label{sec:conc}
We have considered a fully-discrete scheme for 
the magnetic induction equations that arise as a
submodel in the MHD equations of plasma physics. 
In future, our plan is to extend the semi-discrete scheme
given in \cite{uk1} to a semi-implicit fully-discrete
scheme. We would like to show the stabilty of the aforementioned
semi-implicit scheme in case of magnetic induction equations with
resistivity.


\begin{thebibliography}{10}
  

  
\bibitem{BeKr1} N.~Besse and D.~Kr\"oner.  \newblock Convergence of
  the locally divergence free discontinuous Galerkin methods for
  induction equations for the 2D-MHD system.  \newblock {\em M2AN
    Math. Model. Num. Anal} 39(6):1177-1202, 2005.

  
\bibitem{BlBr1} J.U.~Brackbill and D.C.~Barnes.  \newblock The
  effect of nonzero $\Div B$ on the numerical solution of the
  magnetohydrodynamic equations.  \newblock {\em J. Comp. Phys.},
  35:426-430, 1980.

  
\bibitem{DW1} W.~ Dai and P.R.~Woodward.  \newblock A simple finite
  difference scheme for multi-dimensional magnetohydrodynamic
  equations.  \newblock {\em J. Comp. Phys.}, 142(2):331-369, 1998.
  

  
\bibitem{fkrsid1} F.~Fuchs, K.H.~Karlsen, S.~Mishra and N.H.~Risebro.
  \newblock Stable upwind schemes for the Magnetic Induction equation.
  \newblock {\em Preprint,} Submitted.
  


\bibitem{uk} U.~Koley, S.~Mishra, N.H.~Risebro and M.~Sv\"{a}rd.
  \newblock Higher order finite difference schemes for the Magnetic Induction
  equations.  \newblock {\em BIT Numer Math.}, 49: 375-395 (2009).


\bibitem{uk1} U.~Koley, S.~Mishra, N.H.~Risebro and M.~Sv\"{a}rd.
  \newblock Higher order finite difference schemes for the Magnetic Induction
  equations with resistivity.  \newblock {\em Preprint}, Submitted.

  
  
\bibitem{GustafssonKreissOliger}
B.~Gustafsson, H.-O. Kreiss, and J.~Oliger.
\newblock {\em Time dependent problems and difference methods}.
\newblock John Wiley \& Sons, Inc., 1995.



\bibitem{svardsid3} S.~ Mishra and M. Sv\"ard.  \newblock On stability
  of numerical schemes via frozen coefficients and magnetic induction
  equations.  \newblock {\em Preprint}, Submitted.



\bibitem{Svard04} M.~Sv\"{a}rd \newblock 
 On coordinate transformations for summation-by-parts operators
  \newblock {\em J. Sci. Comput.}
  20(2004), 29-42.

\bibitem{SN1} M.~Sv\"{a}rd and J.~Nordstr\"{o}m.  \newblock On the
  order of accuracy for difference approximations of initial-boundary
  value problems.  \newblock {\em Journal of Computational Physics},
  218(2006), 333-352.


  
\bibitem{TF1} M.~Torrilhon and M.~Fey.  \newblock
  Constraint-preserving upwind methods for multidimensional advection
  equations.  \newblock {\em SIAM. J. Num. Anal.}, 42(4):1694-1728,
  2004.



  



\end{thebibliography}
\end{document}